\documentclass{amsart}
\usepackage{graphicx}
\usepackage{amssymb}
\newtheorem{thm}{Theorem}

\newtheorem{lem}[thm]{Lemma}
\newtheorem{prop}[thm]{Proposition}
\theoremstyle{definition}
\newtheorem{defn}[thm]{Definition}
\theoremstyle{remark}

\numberwithin{equation}{section}

\DeclareMathOperator{\Aut}{Aut}

\DeclareMathOperator{\Irr}{Irr}
\DeclareMathOperator{\PI}{PI}
\DeclareMathOperator{\CF}{CF}
\begin{document}

\title{Perfect Isometry Groups for Cyclic Groups of Prime Order}%
\author{Pornrat Ruengrot}%
\address{Science Division, Mahidol University International College\\
999 Phutthamonthon 4 Road, Salaya, Nakhonpathom, Thailand 73170}%
\email{pornrat.rue@mahidol.edu}%


\begin{abstract}
A perfect isometry is an important relation between blocks of finite groups as many information about blocks are preserved by it. If we consider the group of all perfect isometries between a block to itself then this gives another information about the block that is also preserved by a perfect isometry. The structure of this group depends on the block and can be fairly simple or extremely complicated. In this paper we study the perfect isometry group for the block of $C_p$, the cyclic group of prime order, and completely describe the structure of this group. The result shows that any self perfect isometry for $C_p$ is essentially either induced by an element in $\mbox{Aut}(C_p)$, or obtained by multiplication by one of its linear characters, or a composition of both
\end{abstract}
\maketitle
\section{Introduction}

Let $p$ be a prime number. Let $(K,\mathcal{O},k)$ be a $p$-modular system where $\mathcal{O}$ is a
complete local discrete valuation ring with field of fraction $K$ of characteristic 0 and residue class field $k$ of characteristic $p$. We suppose that $K$ is sufficiently large for all groups considered. Let $G$ be a finite group and $B$ and block of $\mathcal{O}G$. Denote by $R_K(B)$ the free abelian group generated by $\Irr(B)$. Let $H$ be another finite group and $A$ a block of $\mathcal{O}H$.

\subsection{Perfect isometries}
\begin{defn}\cite[Definition 1.1]{Broue}
A generalized character $\mu:G\times H\longrightarrow K$ is said to be \emph{perfect} if it satisfies the following two conditions.
\begin{enumerate}
  \item[(i)] [Integrality] For all $(g,h)\in G\times H$ we have
  \[\frac{\mu(g,h)}{|C_G(g)|}\in\mathcal{O}\quad\text{and}\quad\frac{\mu(g,h)}{|C_H(h)|}\in\mathcal{O}.\]
  \item[(ii)] [Separation] If $\mu(g,h)\neq 0$ then $g$ is
  a $p$-regular element of $G$ if and only if $h$ is $p$-regular element of $H$.
\end{enumerate}
\end{defn}

Let $\mu$ be a generalized character of $G\times H$. Then, following Broue \cite{Broue}, $\mu$ defines two linear maps \[I_\mu:R_K(A)\longrightarrow R_K(B)\quad,\quad
R_\mu:R_K(B)\longrightarrow R_K(A)\] defined by
\begin{equation}\label{fml:I_given_mu}
    I_\mu(\beta)(g)=\frac{1}{|H|}\sum_{h\in H}
\mu(g,h^{-1})\beta(h) =\, \langle\mu(g,\cdot),\beta\rangle_H
\end{equation} and
\begin{equation}\label{fml:R_given_mu}
    R_\mu(\alpha)(h)=\frac{1}{|G|}\sum_{g\in G}
\mu(g^{-1},h)\alpha(g) =\, \langle\mu(\cdot,h),\alpha\rangle_G
\end{equation} for $\alpha\in R_K(B)$ and $\beta\in R_K(A)$. Here $\langle\,,\,\rangle_G$ denotes the standard inner product for class functions of $G$. Furthermore, the maps $I_\mu$ and $R_\mu$ are adjoint to each other with respect to $\langle\,,\,\rangle$. That is,
\[\langle I_\mu(\beta),\alpha\rangle_G=\langle\beta,R_\mu(\alpha)\rangle_H\] for all
$\alpha\in  R_K(B),\beta\in  R_K(A)$.

The definition of perfect character $\mu$ can also be stated in terms of conditions on the maps $I_\mu$ and $R_\mu$ as follows.

Let $\CF(G,B;K)$ be the subspace of $\CF(G;K)$ of class functions
generated by $\Irr(B)$. Let
$\CF(G,B;\mathcal{O})$ be the subspace of $\CF(G,B;K)$ of $\mathcal{O}$-valued
class functions. Let $\CF_{p'}(G,B;K)$ be the subspace of class functions
$\alpha\in \CF(G,B;K)$ vanishing outside $p$-regular elements.

If $\mu$ is a generalized character of $G\times H$, the linear maps $I_\mu,R_\mu$ defined in
(\ref{fml:I_given_mu}),(\ref{fml:R_given_mu}) can be extended in the usual
way to the linear maps
\[I_\mu:\CF(H,A;K)\longrightarrow \CF(G,B;K)\quad,\quad R_\mu:\CF(G,A;K)\longrightarrow \CF(H,B;K).\]

\begin{prop}\label{prop_broue4.1}\cite[Proposition 4.1]{Broue} A generalized character $\mu$ is perfect if and only if
\begin{enumerate}
  \item[(i*)] $I_\mu$ gives a map from $\CF(H,A;\mathcal{O})$ to $\CF(G,B;\mathcal{O})$ and $R_\mu$ gives a map from $\CF(G,B;\mathcal{O})$ to $\CF(H,A;\mathcal{O})$.
  \item[(ii*)] $I_\mu$ gives a map from $\CF_{p'}(H,A;K)$ to $\CF_{p'}(G,B;K)$ and $R_\mu$ gives a map from $\CF_{p'}(G,B;K)$ to $\CF_{p'}(H,A;K)$.
\end{enumerate}
\end{prop}

We have seen that a generalized character $\mu$ of $G\times H$ defines a linear map $I_\mu:R_K(A)\longrightarrow R_K(B)$. In fact, any linear map $I:R_K(A)\longrightarrow R_K(B)$ is of the form $I_\mu$ for some generalized character $\mu$ of $G\times H$.

\begin{lem}
  Let $I:R_K(A)\longrightarrow R_K(B)$ be a linear map. Then, there is a generalized character $\mu_I$ of $G\times H$ such that $I=I_{\mu_I}$. Furthermore $\mu_I$ is defined by \[\mu_I(g,h)=\sum_{\chi\in\Irr(A)}I(\chi)(g)\chi(h),\quad\mbox{for all }g\in G,h\in H.\]
\end{lem}

\begin{proof}
Define $\mu_I$ as in the lemma. It is clear that $\mu_I$ is a generalized character of $G\times H$. We will show that $I_{\mu_I}=I$. Let $\beta\in  R_K(A)$ and $g\in G$, then
\begin{eqnarray*}
  I_{\mu_I}(\beta)(g) &=& \frac{1}{|H|}\sum_{h\in H}\mu(g,h^{-1})\beta(h) \\
   &=& \frac{1}{|H|}\sum_{h\in H}\left(\sum_{\varphi\in
   \Irr(A)}I(\varphi)(g)\varphi(h^{-1})\right)\beta(h)\\
   &=&\sum_{\varphi\in \Irr(A)}\frac{1}{|H|}\left(\sum_{h\in
   H}\varphi(h^{-1})\beta(h)\right)I(\varphi)(g)\\
   &=&\sum_{\varphi\in \Irr(A)}\langle\beta,\varphi\rangle I(\varphi)(g)\\
   &=&I\left(\sum_{\varphi\in \Irr(A)}\langle\beta,\varphi\rangle
   \varphi\right)(g)\\
   &=&I(\beta)(g).
\end{eqnarray*}
\end{proof}

\begin{defn}\cite[Definition 1.4]{Broue}
  Let $I:R_K(A)\longrightarrow R_K(B)$ be a linear map. We say that $I$ is a \emph{perfect isometry} if $I$ is an isometry and $I=I_\mu$ where $\mu$ is perfect.
\end{defn}

It turns out that if $I:R_K(A)\longrightarrow R_K(B)$ is an isometry and $I=I_\mu$, then $R_\mu$ is the inverse of $I_\mu$.

\begin{lem}\label{lem:I-iso-R-inverse} Suppose that $I_\mu$ is an isometry. Then $R_\mu$ is also an isometry and $(I_\mu)^{-1}=R_\mu$.
\end{lem}
\begin{proof}
Let $\beta\in \Irr(A)$. Since $I_\mu$ is an isometry, we have $I_\mu(\beta)\in\pm \Irr(B)$. Let $\alpha\in
\Irr(B)$. Since $I_\mu$ and $R_\mu$ are adjoint,
\[\langle I_\mu(\beta),\alpha\rangle_G\,=\,\langle\beta,R_\mu(\alpha)\rangle_H.\]
As the left hand side takes values in $\{0,\pm 1\}$, this implies that $R_\mu(\alpha)\in\pm \Irr(A)$. By adjointness again,\[\langle I_\mu(R_\mu(\alpha)),\alpha\rangle_G\,=\,\langle R_\mu(\alpha),R_\mu(\alpha)\rangle_H\,=1.\] Since $I_\mu(R_\mu(\alpha))\in\pm \Irr(B)$, this forces $I_\mu(R_\mu(\alpha))=\alpha$. Similarly,
\[1=\,\langle I_\mu(\beta),I_\mu(\beta)\rangle_G\,=\,\langle\beta,R_\mu(I_\mu(\beta))\rangle_H\]
implies that $R_\mu(I_\mu(\beta))=\beta$. Hence $R_\mu=(I_\mu)^{-1}$ and
$R_\mu$ is an isometry.
\end{proof}

If $I: R_K(A)\longrightarrow  R_K(B)$ is an isometry, then
$I(\chi)\in\pm\Irr(B)$ for all $\chi\in\Irr(B)$. So $I$ defines a
bijection $I^+:\Irr(A)\longrightarrow\Irr(B)$ and a sign function
$\varepsilon_I:\Irr(A)\longrightarrow\{\pm 1\}$ such that
$I(\chi)=\varepsilon_I(\chi)I^+(\chi)$ for $\chi\in\Irr(A)$. This gives a
bijection with signs between $\Irr(A)$ and $\Irr(B)$. We say that a sign is \emph{all-positive} if $\varepsilon_I(\chi)=+1\,\forall\chi\in\Irr(A)$ and \emph{all-negative} if $\varepsilon_I(\chi)=-1\,\forall\chi\in\Irr(A)$. We also say that a sign is \emph{homogenous} if it is either all-positive or all-negative.

\subsection{Perfect isometry group}
We will now restrict our attention to the case where $A=B$. From Proposition \ref{prop_broue4.1} and Lemma \ref{lem:I-iso-R-inverse} it is clear that if $I,J:R_K(B)\longrightarrow R_K(B)$ are perfect isometries then so are $I^{-1}$ and $I\circ J$. Moreover, the identity map $id:R_K(B)\longrightarrow R_K(B)$ is also a perfect isometry. This leads us to define the following group.

\begin{defn}
  The set of all perfect isometries $I:R_K(B)\longrightarrow R_K(B)$ forms a group under composition. We will call this group the \emph{perfect isometry group} for $B$ and denote it by $\PI(B)$.
\end{defn}

This group is invariant under perfect isometries:

\begin{lem}
  If $A,B$ are any two blocks and there exists a perfect isometry $I:R_K(A)\longrightarrow R_K(B)$, then $\PI(A)\cong\PI(B)$.
\end{lem}

\begin{proof}
One can check that the map $I:J\mapsto I\circ J\circ I^{-1}$ gives a desired isomorphism $\PI(A)\cong\PI(B)$.
\end{proof}

\section{Perfect Isometry Group for $C_p$}
In this section we will study perfect isometry group for $G=C_p$ where $C_p$ is the cyclic group of prime order $p$.  Let
$\zeta = e^{2\pi i/p}$ be a primitive $p$-root of unity and let $\mathcal{O}=\mathbb{Z}_p(\zeta)$ and $K=\mathbb{Q}_p(\zeta)$. Since $G$ is a $p$-group, there is only a single block $B=\mathcal{O}G$.

Before stating the main theorem about perfect isometry group for $B$, we will first define the following isometries in $R_K(B)$ that are crucial in the structure of $\PI(B)$.
\begin{itemize}
  \item Let $\lambda\in\Irr (B)$. Since $\lambda$ is a linear character, $\lambda\chi\in\Irr(G)$ for every $\chi\in\Irr(B)$ where $\lambda\chi(g) =\lambda(g)\chi(g), \forall g\in G$. Thus $\lambda$ induces and isometry
$I:\chi\mapsto\lambda\chi$ defined by $I_\lambda(\chi)=\lambda\chi$ for $\chi\in\Irr(B)$.
  \item Let $\sigma\in\Aut(G)$. Then $\sigma$ acts on $R_K(B)$ via $\theta^\sigma(h)=\theta(h^{\sigma^{-1}})$. Since the action by $\sigma$ permutes elements in the set $\Irr(B)$, this induces an isometry $I_{\sigma}:R_K(B)\longrightarrow R_K(B)$ defined by $I_{\sigma}(\chi)=\chi^\sigma$ for $\chi\in\Irr(B)$.
\end{itemize}

Observes that the each isometry $I_\lambda, I_\sigma$ has all-positive sign.

The main result can now be stated as follows.
\begin{thm}\label{thm:PI-gp-of-C_p} Let $G$ be a cyclic group of order $p$. Let $B=\mathcal{O}G$ be the block of $G$. Then
  \begin{enumerate}
    \item Every perfect isometry in $\PI(B)$ has a homogenous sign.
    \item Every perfect isometry in $\PI(B)$ with all-positive sign is a composition of isometries of the following forms:
    \begin{itemize}
      \item $I_\lambda:\chi\mapsto \lambda\chi, \forall \chi\in\Irr(B)$ where $\lambda\in\Irr(B)$.
      \item $I_\sigma:\chi\mapsto \chi^\sigma, \forall \chi\in\Irr(B)$ where $\sigma\in\Aut(G)$.
    \end{itemize}
    \item We have $$\PI(B)\cong (G\rtimes \Aut(G))\times\langle -id\rangle.$$
  \end{enumerate}
\end{thm}

\subsection{Proof of the main theorem}
Suppose $G=\langle g\rangle$. Then we can write $\Irr(B)=\{\chi_0,\chi_1,\ldots,\chi_{p-1}\}$ where $\chi_a(g^b)=\zeta^{ab}$ (so $\chi_0$ is the trivial character).

\begin{table}[!ht]
\begin{center}
\caption{Character table of $C_p$}\label{tab:char_tab_c_p}
\begin{tabular}{|c|c|c|c|c|c|c|c|}
  \hline
   $C_q$& 1 & $g$ & $g^2$ & \ldots &$g^m$ &\ldots &$g^{p-1}$ \\
   \hline
  $\chi_0$ & 1 & 1 & 1 & \ldots & 1&\ldots &1 \\
  $\chi_1$ & 1 & $\zeta$ & $\zeta^2$ & \ldots & $\zeta^m$ &\ldots & $\zeta^{p-1}$ \\
  $\chi_2$ & 1 & $\zeta^2$ & $\zeta^4$ & \ldots & $\zeta^{2m}$&\ldots &$\zeta^{2(p-1)}$ \\
  \ldots &  & &  &  & & & \\
  $\chi_k$ & 1 & $\zeta^k$ & $\zeta^{2k}$ & \ldots & $\zeta^{mk}$&\ldots &$\zeta^{k(p-1)}$ \\
  \ldots &  & &  &  & & & \\
  $\chi_{p-1}$ & 1 & $\zeta^{p-1}$ & $\zeta^{2(p-1)}$ & \ldots & $\zeta^{m(p-1)}$&\ldots &$\zeta^{(p-1)^2}$ \\
  \hline
\end{tabular}
\end{center}
\end{table}

We will first show that any perfect isometry $I\in\PI(B)$ has a homogenous sign.
\begin{lem}\label{lem:homog-sign}
Let $I\in\PI(B)$. Then
\begin{enumerate}
\renewcommand{\theenumi}{\roman{enumi}}
\renewcommand{\labelenumi}{(\theenumi)}
  \item $I$ has a homogenous sign.
  \item Either $I(\chi)(1) = \chi(1)\,\forall \chi\in\Irr(G)$ or $I(\chi)(1) = -\chi(1)\,\forall \chi\in\Irr(G)$.
\end{enumerate}
\end{lem}

\begin{proof}
By \cite[Lemma 1.6]{Broue}, we know that $I(\chi)(1)/\chi(1)$ is an invertible element in $\mathcal{O}$ for any $\chi\in\Irr(B)$. Since both $I(\chi)(1)$ and $\chi(1)$ are powers of $p$, we must have $I(\chi)(1)=\pm\chi(1)$. Consider \[\frac{\mu_I(1,1)}{|G|} =\frac{\sum_{\chi\in\Irr(B)}I(\chi)(1)\chi(1)}{p} =\frac{\sum_{\chi\in\Irr(B)}(\pm\chi(1)^2)}{p}.\] Since $\mu_I(1,1)/|G|\in\mathcal{O}$ and $|\sum_{\chi\in\Irr(B)}(\pm\chi(1)^2)|\le p$ this means that either
\begin{itemize}
  \item $\sum_{\chi\in\Irr(B)}(\pm\chi(1)^2)=p$ in which case all the signs are positive, or
  \item $\sum_{\chi\in\Irr(B)}(\pm\chi(1)^2)=-p$ in which case all the signs are negative.
\end{itemize} This proves $(i)$, and $(ii)$ follows from $I(\chi)(1)=\pm\chi(1)$ by above.
\end{proof}

This proves part 1 of the main theorem. Since an isometry $I$ is perfect if and only if $-I$ is perfect, it suffices to consider perfect isometries with all-positive sign. We will show that these isometries are precisely compositions of isometries of the form $I_\lambda$ and $I_\sigma$ for $\lambda\in\Irr(B)$ and $\sigma\in\Aut(G)$. First, we need to show that $I_\lambda$ is perfect for any $\lambda\in\Irr(B)$.

\begin{lem}\label{lem:PI:mult-by-linear-is-perfect}
Let $\lambda\in\Irr(B)$. Then $I_\lambda$ is a perfect isometry.
\end{lem}

\begin{proof}
Suppose $\lambda=\chi_a$. Consider
\begin{align*}
  \mu_{I_\lambda}(g^m,g^n) &=\sum_{r=0}^{p-1}I_\lambda(\chi_r)(g^m)\chi_r(g^n) = \sum_{r=0}^{p-1}\zeta^{m(a+r)+rn}\\
  &=\zeta^{ma}\sum_{r=0}^{p-1}\zeta^{rm+rn}=\zeta^{ma}\mu_{id}(g^m,g^n).
\end{align*} Since the identity map is perfect. It is clear that $I_\lambda$ is also perfect.
\end{proof}

The character multiplication makes $\Irr(B)$ into a cyclic group generated by $\chi_1$.

\begin{lem}\label{lem:hom-from-linear-char-to-PI}
The map $\lambda\mapsto I_\lambda$ is a group monomorphism from $\Irr(B)$ into $\PI(B)$.
\end{lem}
\begin{proof}
Let $\lambda, \lambda'\in\Irr(B)$. For each $\chi\in\Irr(B)$ we have \[I_{\lambda \lambda'}(\chi)=\lambda \lambda' \chi=\lambda  I_{\lambda'}(\chi)=I_\lambda(I_{\lambda'}(\chi))=(I_\lambda\circ I_{\lambda'})(\chi).\] This shows that the map $\lambda\mapsto I_\lambda$ is a group homomorphism. Let $\lambda$ be in the kernel. Then $I_\lambda(\chi_0)=\chi_0=\lambda\chi_0$. So $\lambda=\chi_0$.
\end{proof}

\begin{lem}\label{lem:sum_eq_0}
Let $m$ be an integer not divisible by $p$. Then
\[\sum_{k=1}^p\zeta^{km}=0.\]
\end{lem}

\begin{proof}
Let $P(X)=X^{p-1}+X^{p-2}+\ldots+X+1$. Since $\zeta^m$ is a
$p$th-root of unity which is not equal to 1, $P(\zeta^m)=0$.
Consider
\begin{align*}
XP(X) &=X^p+X^{p-1}+\ldots +X^2+X\\
X^mP(X^m) &=X^{pm}+X^{m(p-1)}+\ldots +X^{2m}+X^m.
\end{align*} Hence \[0=\zeta^mP(\zeta^m)=\sum_{k=1}^p\zeta^{km}.\]
\end{proof}

Next, we will show that the isometry $I_\sigma$ is perfect for any $\sigma\in\Aut(G)$.

\begin{lem}
Let $\sigma\in \Aut(G)$. Then $I_{\sigma}$ is a perfect isometry.
\end{lem}

\begin{proof}
Since $\sigma$ is an automorphism, $g^{\sigma^{-1}}=g^a$ for some integer $a$ not divisible by $p$. Now
\begin{align*}
  \mu_{I_\sigma}(g^m,g^n) &= \sum_{r=0}^{p-1}I_\sigma(\chi_r)(g^m)\chi_r(g^n)=\sum_{r=0}^{p-1}\chi_r(g^{am})\chi_r(g^n)\\
  &=\mu_{id}(g^{am},g^n).
\end{align*} Since $g^{am}$ is $p$-regular if and only if $g^m$ is $p$-regular, and the identity map is perfect, it is clear that $I_\sigma$ is a perfect isometry.
\end{proof}

\begin{lem}\label{lem-hom-from-autG-to-PIB}
The map $\sigma\mapsto I_{\sigma^{-1}}$ is a group monomorphism from $\Aut(G)$ into $\PI(B)$.
\end{lem}
\begin{proof}
Let $\sigma, \tau\in\Aut(G)$. For each $\chi\in\Irr(B)$ we have \[I_{(\sigma\tau)^{-1}}(\chi) =\chi^{(\sigma\tau)^-1} =\chi^{\tau^-1\sigma^-1}= I_{\sigma^{-1}}(\chi^{\tau^-1}) =I_{\sigma^{-1}}(I_{\tau^{-1}}(\chi)) =(I_{\sigma^{-1}}\circ I_{\tau^{-1}})(\chi).\] This shows that the map $\sigma\mapsto I_{\sigma^{-1}}$ is a group homomorphism. Suppose $\sigma\in\Aut(G)$ is in the kernel. Then $\chi_1(g)= I_\sigma(\chi_1)(g)=\chi_1^\sigma(g)=\chi_1(g^{\sigma^{-1}})$. This implies that $\sigma$ is the identity map.
\end{proof}

Since $I_\lambda, I_\sigma$ are perfect isometries with all-positive sign for any $\lambda\in\Irr(B)$ and any $\sigma\in\Aut(G)$. A composition of $I_\lambda, I_\sigma$ is also a perfect isometry with all-positive sign. Before proving the converse, we need the following lemma.

\begin{lem}\label{lem:sum_omega_condition}
Let $\alpha = C_0+C_1\zeta+\cdots+C_{p-1}\zeta^{p-1}$ where $C_i\in\mathbb{Z}$ for all $i$. Suppose that $\alpha/p\in\mathcal{O}$. Then $C_i\equiv C_j \mod p$ for all $i,j$.
\end{lem}

\begin{proof}
Since $\alpha/p\in\mathbb{Z}_p(\zeta)$ , we can write $\alpha = pa_0 +pa_1\zeta+\cdots +pa_{p-2}\zeta^{p-2}$ where $a_i\in\mathbb{Z}_p \forall i$, as $\{1,\zeta,\ldots,\zeta^{p-2}\}$ is a basis of $\mathbb{Z}_p(\zeta)$ over $\mathbb{Z}_p$. Write $\zeta^{p-1}=-1-\zeta-\cdots -\zeta^{p-2}$. Then
\begin{align*}
   \alpha &= (C_0-C_{p-1}) + (C_1-C_{p-1})\zeta +\cdots + (C_{p-2}-C_{p-1})\zeta^{p-2}  \\
    &= pa_0 +pa_1\zeta+\cdots +pa_{p-2}\zeta^{p-2}.
\end{align*} Comparing the coefficients, we have $C_i-C_{p-1}\in p\mathbb{Z}_p$ for $i=1,\ldots,p-2$. But $C_i-C_{p-1}$ are integers, we have $C_i-C_{p-1}\in p\mathbb{N}$ and so $C_i\equiv C_j \mod p$ for all $i,j$ as claimed.
\end{proof}

\begin{lem}\label{lem:all-pos-then-LA}
  Let $I\in\PI(B)$ be a perfect isometry with all-positive sign. Then $I$ is a composition of $I_\lambda$ and $I_\sigma$ for some $\lambda\in\Irr(B)$ and $\sigma\in\Aut(G)$.
\end{lem}

\begin{proof}
Suppose $I$ does not fix the trivial character, say $I(\chi_0)=\chi_a$ for some $a\in\{1,\ldots,p-1\}$. Then $(I_{\chi_{p-a}}\circ I)(\chi_0)=\chi_{p-a}\chi_a=\chi_0$. So, by composing with $I_\lambda$ for a suitable $\lambda\in\Irr(B)$, we can assume that $I$ fixes the trivial character. Suppose now that $I(\chi_1)=\chi_b$ for some $b\in\{1,\ldots,p-1\}$. Let $\sigma$ be the automorphism $g\mapsto g^b$. Then
\begin{align*}
  (I_\sigma\circ I)(\chi_1)(g)&=I_\sigma(\chi_b)(g)=\chi_b(g^{b^{-1}})=\zeta=\chi_1(g)\\
  (I_\sigma\circ I)(\chi_0)(g)&=I_\sigma(\chi_0)(g)=1=\chi_0(g).
\end{align*} Thus $I_\sigma\circ I$ fixes both $\chi_0$ and $\chi_1$. So, by composing with $I_\sigma, I_\lambda$ for suitable $\lambda\in\Irr(B), \sigma\in\Aut(G)$, we can also assume that $I$ fixes $\chi_0,\chi_1$. The lemma is proved once we show that $I$ must then be the identity map. To see this, consider
\begin{align*}
  \mu_I(g,g^{-1})&=\sum_{r=0}^{p-1}I(\chi_r)(g)\chi_r(g^{-1})\\
  &= 1+1+\sum_{r=2}^{p-1}\frac{I(\chi_r)(g)}{\chi_r(g)}\\
  &= C_0+C_1\zeta+\cdots+C_{p-1}\zeta^{p-1}
\end{align*} where $0\le C_i\le p$ for all $i$ and $C_0\ge 2$ ($C_i$ is the number of occurrences of $\zeta^i$ in $1+1+\sum_{r=2}^{p-1}\frac{I(\chi_r)(g)}{\chi_r(g)}$. Since $I$ is perfect, $\mu_I(g,g^{-1})\in p\mathcal{O}$. So, by Lemma \ref{lem:sum_omega_condition}, we have $C_i\equiv C_j \mod p$ for all $i,j$. But $C_0\ge 2$. So we must have $C_0=p$ and $C_i=0$ for all $i\ne 0$. This means that $\frac{I(\chi_r)(g)}{\chi_r(g)}=1$ for all $r$. Thus $I(\chi_r)=\chi_r$ for all $r$ and so $I$ is the identity map.
\end{proof}

This proves part 2 of the main theorem.

Let $\mathcal{L}$ and $\mathcal{A}$ be the images of the monomorphisms $\Irr(B)\longrightarrow\PI(B)$ and $\Aut(G)\longrightarrow\PI(B)$ respectively. Since character multiplication is an abelian operation, we observe that $\mathcal{L}$ is an abelian group.

\begin{lem}
  The group $\PI(B)$ contains the subgroup $\mathcal{L}\rtimes\mathcal{A}$.
\end{lem}

\begin{proof}
For any $\lambda\in\Irr(B),\sigma\in\Aut(G)$ and $\chi\in\Irr(B)$, we have
\begin{align*}
  (I_\sigma\circ I_\lambda\circ(I_\sigma)^{-1})(\chi) &=  (I_\sigma\circ I_\lambda)(\chi^{\sigma^{-1}})=I_\sigma(\lambda\chi^{\sigma^{-1}})=\lambda^\sigma\chi=I_{\lambda^\sigma}(\chi).
\end{align*} Thus $\mathcal{L}$ is normal in $\mathcal{A}$. Since $\mathcal{L}$ is abelian, $\mathcal{L}$ is also normal in $\mathcal{L}\mathcal{A}$. Suppose $I\in \mathcal{L}\cap\mathcal{A}$, say $I=I_\lambda=I_\sigma$ for some $\lambda\in\Irr(B)$ and $\sigma\in\Aut(G)$. Then $I_\lambda(\chi_0)=I_\sigma(\chi_0)$. This implies that $\lambda=\chi_0$ and so $I$ is the identity map. Hence $\mathcal{L}$ intersects $\mathcal{A}$ trivially and so $\mathcal{L}\mathcal{A}=\mathcal{L}\rtimes\mathcal{A}$ is a subgroup of $\PI(B)$.
\end{proof}

We will now prove part 3 of the main theorem.

\begin{lem}
  We have $$\PI(B)= \mathcal{L}\rtimes\mathcal{A}\times\langle -id\rangle$$ and $$\PI(B)\cong G\rtimes\Aut(G)\times\langle -id\rangle.$$
\end{lem}

\begin{proof}
By Lemma \ref{lem:homog-sign} every perfect isometry $I\in\PI(B)$ has a homogenous sign. Thus $\PI(B)=S\times\langle -id\rangle$ where $S$ is a subgroup containing all perfect isometries with all-positive sign. Since perfect isometries in $\mathcal{L}\rtimes\mathcal{A}$ has all-positive sign, we have $\mathcal{L}\rtimes\mathcal{A}\le S$. But Lemma \ref{lem:all-pos-then-LA} says that any perfect isometry with all-positive sign must be in $\mathcal{L}\rtimes\mathcal{A}$, we have $S\le\mathcal{L}\rtimes\mathcal{A}$ and so $S=\mathcal{L}\rtimes\mathcal{A}$. Finally, since $\Irr(B)\longrightarrow\PI(B)$ and $\Aut(G)\longrightarrow\PI(B)$ are monomorphisms, it is clear that $\mathcal{L}\cong\Irr(B)\cong G$ and $\mathcal{A}\cong\Aut(G)$.

\end{proof}

\bibliographystyle{amsplain}

\begin{thebibliography}{10}
    \bibitem{Broue} M. Broue, {\it Isometries parfaites, types de blocs, categories derivees}, Asterisque, 181-182 (1990), 61--92.
\end{thebibliography}

\end{document}